\documentclass{imanum}
\usepackage{graphicx}
\usepackage{mathrsfs}
\usepackage{color}

\newtheorem{tm}{Theorem}[section]
\newtheorem{rk}{Remark}[section]
\newtheorem{ap}{Assumption}[section]
\newtheorem{df}{Definition}[section]
\newtheorem{prop}{Proposition}[section]
\newtheorem{lm}{Lemma}[section]

\newcommand{\E}{\mathbb E}
\newcommand{\PP}{\mathbb P}
\newcommand{\N}{\mathbb N}
\newcommand{\R}{\mathbb R}

\newcommand{\DD}{\mathcal D}

\newcommand{\RR}{\mathcal R}

\newcommand{\FFF}{\mathscr F}

\jno{drnxxx}
\received{Oct. 07, 2015}
\revised{Dec. 14, 2016}
\accepted{Jan. 6, 2017}

\begin{document}

\title{Finite element approximations for second order stochastic differential
equation driven by fractional Brownian motion}
\shorttitle{Finite element approximations for second order SDE with fBm}

\author{
{\sc Yanzhao Cao}\thanks{yzc0009@auburn.edu}\\[2pt]
Department of Mathematics and Statistics, Auburn University, Auburn, AL 36849
\\[2pt]
{\sc Jialin Hong\thanks{hjl@lsec.cc.ac.cn}
and
Zhihui Liu}\thanks{Corresponding author. Email: liuzhihui@lsec.cc.ac.cn}\\[2pt]
Academy of Mathematics and Systems Science, Chinese Academy of Sciences, Beijing, China
}

\shortauthorlist{Yanzhao Cao, Jialin Hong and Zhihui Liu}

\maketitle

\begin{abstract}
{We consider finite element approximations for a one dimensional second order stochastic differential equation of boundary value type driven by a fractional Brownian motion with Hurst index $H\le 1/2$. 
We make use of a sequence of approximate solutions with the fractional noise replaced by its piecewise constant approximations to construct the finite element approximations for the equation. 
The error estimate of the approximations is derived through rigorous convergence analysis.}
{stochastic differential equation of boundary value type, 
fractional Brownian motion, 
piecewise constant approximation,
finite element approximation}
\end{abstract}

\section{Introduction}

Many physical and engineering phenomena can be modeled by stochastic differential equations (SDEs) and stochastic partial differential equations (SPDEs) when including some levels of uncertainties. 
The advantage of modeling with stochastic equations is that they are able to more fully capture the practical behavior of underlying models; it also means that the corresponding numerical analysis will require new tools to simulate the systems, produce the solutions, and analyze the information stored within the solutions.
Stochastic equations derived from fluid flows and other engineering fields are often driven by white noise. 
The white noise is an uncorrelated noise with delta function as its covariance. 
However, random fluctuations in complex systems may not be uncorrelated, i.e., they may not be white noise. Such noises are named as colored noises.

As an important class of colored noises, the fractional Brownian motion (fBm) type noise appears naturally in the modeling of many physical and social phenomena \citep[see, e.g.,][]{MN68, WGBS15}. 
For examples, fBm is suitable in describing the widths of consecutive annual rings of a tree and the temperature at a specific place \citep[see, e.g.,][]{Shi99}; it can also be applied to simulate the turbulence in an incompressible fluid flow and the prices of electricity in a liberated electricity market \citep[see, e.g.,][]{Sim03}.
As a centered Gaussian process, the fBm can be defined as follows. 
Let $\DD=(0,1)$ and denote $\overline \DD$, $\partial \DD$ the closure and the boundary of $\DD$, respectively.
The fBm $W=\{W(x),\ x\in \overline \DD\}$ on $\overline \DD$ is determined by its covariance function
\begin{align}\label{fbm}
\text{Cov}(x,y):
=\E\left[W(x)W(y)\right]
=\frac{x^{2H}+y^{2H}-|x-y|^{2H}}{2},\quad x,y\in \overline \DD. 
\end{align}
Here $H\in (0,1)$ is the so-called Hurst index. 

The fBms with $H<1/2$ and $H>1/2$ are significantly different both physically and mathematically. 
In the first two aforementioned applications, the corresponding fBms have Hurst index $H>1/2$. 
In such cases, the physical process presents an aggregation and persistent behavior. 
On the other hand,  the Hurst index $H$ is less than $1/2$ in the last two cases where the process is anti-persistent and may have long-range negative interactions. 
These two classes of fBms are separated by the standard Brownian motion whose Hurst index is $H=1/2$.
Mathematically, the fBm with $H>1/2$ is a Gaussian process whose covariance function has a bounded variation on $\overline \DD\times \overline \DD$. 
The stochastic integral against the fBm with $H>1/2$ can be viewed as a pathwise  Riemann-Stieltjes integral (or Young integral) and classical methods are applicative. 
On the contrary, the covariance function of the fBm with $H\le 1/2$ does not have bounded variation. 
This posts a particular difficulty when studying SDEs or SPDEs driven by such noises.

The main objective of this paper is to investigate the well-posedness and finite element approximations for the following second order SDE of boundary type driven by an fBm with $H\le 1/2$:
\begin{align}\label{ell}
\begin{split}
-\frac{d^2}{dx^2}u(x)+f(x,u(x))
&=g(x)+\dot{W}(x),\quad x\in \DD, \\
u(x)&=0,\quad x\in \partial \DD.
\end{split}
\end{align}
Here $g: \DD\rightarrow \R$ is square integrable, $f: \DD\times\R \rightarrow \R$ satisfies certain conditions given in Section \ref{sec2} and $W=\{W(x):\ x\in \overline \DD\}$ is an fBm, determined by \eqref{fbm} with $H\le 1/2$, on a filtered probability space $(\Omega, \FFF, (\FFF_x)_{x\in \overline \DD}, \PP)$. 
The homogenous Dirichlet boundary condition in Eq. \eqref{ell} corresponds to a second order SDE conditioned to hit a particular point at ``time" $x=1$.
As such it is a generalization to general fractional noise of the conditioned diffusions studied in \citep[]{HSV11}. 
On the other hand, this equation can be considered as an elliptic SPDE in one-dimension. 
Note that one can also study homogenous Neumann boundary condition and the main results of this paper are also valid.

Eq. \eqref{ell} driven by the white noise, i.e., $H=1/2$, has been considered by several authors \citep[see, e.g.,][]{ANZ98, CYY07, CHL15, DZ02, GM06, MS06, ZTRK15}. 
\citep[]{ANZ98} investigated the finite difference and finite element approximations of the linear case of Eq. \eqref{ell}. 
They proved the first order convergence for both the finite difference and finite element approximations. 
The three authors in \citep[]{CHL15} investigated the finite element approximations of Eq. \eqref{ell} in possibly any dimensions formulated in the form of Karhunen-Lo\`eve expansions for certain Gaussian noises. 
For the case where $H>1/2$, the well-posedness and finite difference approximations can be studied using the methodology of \citep[]{MS06} by treating the fBm as a colored noise with a special Riesz kernel. 
However, the method of treating the white noise and more regular noises does not apply to fractional noise with $H<1/2$, since the exact solution is less regular. 
To the best of our knowledge, there have not been literatures studying numerical approximations for SDEs or SPDEs driven by fractional noises with $H<1/2$. 

The primary challenge in studying the finite element approximations of Eq. \eqref{ell} driven by fBm with $H<1/2$ is three-folds:
(i) as a colored noise, the increments of the fBm in two disjoint intervals are not independent; 
(ii) the regularity of $\dot{W}$ with $H<1/2$ is very low; 
(iii) the approach of Karhunen-Lo\`{e}ve expansions used in \citep[]{CHL15} fails.

In this paper we study the well-posedness and the finite element approximations of Eq. \eqref{ell} through a special It\^{o} isometry which is only valid for $H\le 1/2$ (see \eqref{ito}). 
Using this isometry we obtain the existence of a unique solution for Eq. \eqref{ell} by analyzing the convergence of a sequence of approximate solutions of SPDEs with the fractional noise replaced by a sum of tensor products between correlated Gaussian random variables and piecewise constant functions in the physical domain. 

Following the well-posedness analysis, we construct the finite element approximations of  Eq. \eqref{ell} through two steps.
In the first step, we derive an error estimate between the exact solution and its approximations which are used in the well-posedness analysis. 
This error estimate also heavily depends on the aforementioned It\^{o} isometry.  
In the second step, we apply the Galerkin finite element method to the approximate noise driven SPDE and obtain the overall error estimate of the finite element solution through an finite element error estimate for the approximate SPDE.

The paper is organized as follows. 
First we define the weak solution and mild solution of Eq. \eqref{ell} and establish their existence and uniqueness in Section \ref{sec2}. 
Next in Section \ref{sec3} we derive the error estimate between the exact solution of Eq. \eqref{ell} and the solution of the approximate SPDE. 
In Section \ref{sec4}, we apply a finite element method to this approximate SPDE and derive the overall error estimate of the finite element solution.
Finally a few concluding remarks are given in Section \ref{sec5}. 

We end this section by introducing some notations which will be used throughout the paper. 
Denote by $L^2(\DD)$ the space of square integrable functions in $\DD$ with its inner product and norm denoted by $(\cdot,\cdot)$ and $\|\cdot\|$, respectively.  
For $r>0$, we use $H^r(\DD)$ to denote the usual Sobolev space whose norm is denoted by $\|\cdot\|_r$. 
We also use $H^1_0(\DD)$ to denote the subspace of $H^1(\DD)$ whose elements vanish on $\partial \DD$. 
We denote by $C$ a generic positive constant independent of either the truncation number $n$ or the grid size $h$ which will changes from one line to another.

\section{Well-posedness of the problem}
\label{sec2}

In this section, we define the weak solution and mild solution of Eq. \eqref{ell} and then establish their equivalence, existence and uniqueness.

\begin{df}\label{df-weak}
An $\FFF_x$-adapted stochastic process $u=\{u(x):x\in \DD\}$ is called a weak solution of Eq. \eqref{ell}, if for every $\phi\in C^2(\DD)\cap C(\overline \DD)$ vanishing on $\partial \DD$ it holds a.s. that
\begin{align}\label{weak}
-\int_{\DD} u(x)\phi''(x)dx+\int_{\DD} f(x,u(x))\phi(x)dx=\int_{\DD} g(x)\phi(x)dx+\int_{\DD} \phi(x)dW(x).
\end{align}
\end{df}

\begin{df}\label{df-mild}
An $\FFF_x$-adapted stochastic process $u=\{u(x):x\in \DD\}$ is called a mild solution of  Eq. \eqref{ell}, if for all $x\in \DD$ it holds a.s. that
\begin{align} \label{mild}
u(x)+\int_{\DD} G(x,y)f(y,u(y))dy=\int_{\DD} G(x,y)g(y)dy+\int_{\DD} G(x,y)dW(y),
\end{align}
where $G$ is the Green's function associated with the Possion equation
with Dirichlet boundary.
\end{df}

It is well known that the related Green's function $G$ is given by $G(x,y)=x\wedge y-xy$, $x,y\in \overline \DD$. 
Obviously  $G$ is Lipschitz continuous over $\overline \DD\times \overline \DD$. 
Without loss of generality, we assume that $f(x,0)=0$ for any $x\in \DD$. Otherwise, we simply replace $f(x,r)$ by $f(x,r)-f(x,0)$ and $g(x)$ by $g(x)-f(x,0)$. 
Assume furthermore that $f$ satisfies the following assumptions.

\begin{ap}\label{ap}
\begin{enumerate}
\item
(Monotone type condition) 
There exists a positive constant $L<\gamma$ such that
\begin{align}\label{one-lip}
(f(x,r)-f(x,s),r-s)\ge -L|r-s|^2,\quad \forall\ x\in \DD,\ r,s\in \R,
\end{align}
where $\gamma$ is the positive constant in the Poincar\'{e} inequality \citep[see, e.g.,][Theorem 6.30]{AF03}:
\begin{align}\label{poin}
\Big\|\frac{d}{dx} v \Big\|^2
\ge \gamma\|v\|^2,\quad \forall\ v\in H^1_0(\DD).
\end{align}

\item  (Linear growth condition) There exists a positive constant $\beta$ such that
\begin{align}\label{lin-gro}
|f(x,r)-f(x,s)|\le\beta(1+|r-s|),\quad \forall\ x\in \DD,\ r,s\in \R.
\end{align}
\end{enumerate}
\end{ap}

We remark that these two conditions can be satisfied when $f$ is a sum of a non-decreasing bounded function and a Lipschitz continuous function with the Lipschitz constant less than $\gamma$ \citep[see, e.g.,][]{BP90,GM06}.
In the case $\DD=(0,1)$ it can be easily shown that $\gamma =2$.
Therefore we assume that $L<2$ throughout the rest of this paper.

Before establishing the well-posedness of Eq. \eqref{ell}, we follow the approach of \citep[]{BJ06} to define  stochastic integral with respect to the fBm $W$ with $H<1/2$. 
To this end, we introduce
the set $\Phi$ of all  step functions on $\DD$ of the form
\begin{align*}
f=\sum_{j=0}^{N-1} f_j\chi_{(a_j,a_{j+1}]},
\end{align*}
where $0=a_0<a_1<\cdots<a_N=1$ is a partition of $\DD$ and $f_j\in\R$, $j=0,1,\cdots,N-1$, $N\in \N_+$. 
For $f\in \Phi$, we define its integral with respect to $W$ by Riemann sum as
\begin{align*}
I(f)=\sum_{j=0}^{N-1} f_j(W(a_{j+1})-W(a_j)),
\end{align*}
and for $f,g\in \Phi$, we define their scalar product as
\begin{align*}
\Psi(f,g):=\E\left[I(f)I(g)\right].
\end{align*}

Next we extend $\Phi$ through completion to a Hilbert space,  denoted by $\Phi^H$. 
By \citep[Lemma 2.1]{BJ06},
we have a characterization of $\Phi^H$ through It\^o isometry for simple functions:
\begin{align}\label{ito}
\Psi(f,g)
&=\frac{H(1-2H)}{2}\int_{\DD}\int_{\DD} 
\frac{(f(x)-f(y))(g(x)-g(y))}{|x-y|^{2-2H}}dxdy    \nonumber  \\
&\quad +H\int_{\DD} f(x)g(x) (x^{2H-1}+(1-x)^{2H-1}) dx,
\quad \forall\ f,g\in \Phi.
\end{align}
This shows that
\begin{align*}
\Phi^H=\Bigg\{f\in L^2(\overline \DD):\int_{\R}\int_{\R} 
\frac{|\overline{f}(x)-\overline{f}(y)|^2}{|x-y|^{2-2H}}dxdy<\infty\Bigg\},
\end{align*}
where $\overline{f}(x)=f(x)$ when $x\in\overline \DD$ and $\overline{f}(x)=0$ otherwise.
As a consequence, the integral $I$ for a measurable deterministic function $f:\overline \DD\rightarrow \R$ with respect to the fBm $W$ is an isometry between $\Phi^H$ and a subspace of $L^2(\PP)$.

\begin{lm}
\begin{enumerate}
\item
The stochastic process $\{v(x):=\int_{\DD} G(x,y)dW(y),\ x\in\overline \DD\}$ possesses an a.s. continuous modification.

\item
Definitions \ref{df-weak} and \ref{df-mild} are equivalent to each other.
\end{enumerate}
\end{lm}

\begin{proof}Let $x_1,x_2\in \DD$. The Ito's isometry \eqref{ito} yields
\begin{align*}
\E\left[|v(x_1)-v(x_2)|^2 \right]  
&=\frac{H(1-2H)}{2}\int_{\DD}\int_{\DD} \frac{|[G(x_1,y)-G(x_2,y)]-[G(x_1,z)-G(x_2,z)]|^2}{|y-z|^{2-2H}} dydz   \\
&\quad+H\int_{\DD} |G(x_1,y)-G(x_2,y)|^2(y^{2H-1}+(1-y)^{2H-1})dx.
\end{align*}
Since $G=\{G(x,y):\ x,y\in \overline \DD\}$ is Lipschitz continuous with respect to both $x$ and $y$, we have
\begin{align*}
|[G(x_1,y)-G(x_2,y)]-[G(x_1,z)-G(x_2,z)]|^2
\le 2|x_1-x_2|\times 2|y-z|.
\end{align*}
Direct calculations yield that
\begin{align*}
\int_{\DD}\int_{\DD} |y-z|^{2H-1}dydz=H(1+2H).
\end{align*}
Therefore, there exists $C=C(H)$ such that 
\begin{align*}
\E\left[|v(x_1)-v(x_2)|^2\right]\le C|x_1-x_2|,\quad x_1,x_2\in \overline \DD,
\end{align*}
from which and the fact that $v$ is Gaussian we conclude that $v$ has an a.s. continuous modification \citep[see, e.g.,][Exercise 4.9]{Kho09}.

Assume  that $u$ satisfies \eqref{mild} and let $\phi\in C_0^\infty(\DD)$. Multiplying  \eqref{mild} by $\phi''(x)$,  integrating over $\DD$, and using the identity $-\int_{\DD} G(x,y) \phi''(y)dy=\phi(x)$,
we obtain \eqref{weak} for smooth $\phi$. The general case follows from the fact that $C_0^\infty(\DD)$ is dense in $C^2(\DD)\cap C(\overline \DD)$.

Suppose now that $u$ satisfies \eqref{weak}. 
Choose $\phi(x)=-\int_{\DD} G(x,y)\psi(y)dy$ with $\psi\in C^\infty(\DD)$. 
Then $\phi\in C^2(\DD)\cap C(\overline \DD)$ vanishing on the boundary $\partial \DD$ and $-\phi''(x)=\psi(x)$.
We conclude
\small{
\begin{align*}
\int_{\DD} u(x)\psi(x)dx+\int_{\DD}\int_{\DD} G(x,y)f(y,u(y))\psi(x)dxdy
=\int_{\DD}\int_{\DD} G(x,y)g(y)\psi(x)dxdy+\int_{\DD}\int_{\DD} G(x,y)\psi(x)dW(y)dx,
\end{align*}}
from which \eqref{mild} follows. The proof is complete.
\end{proof} \\

Next we define a sequence of approximations to the fractional noise $\dot{W}$.
Let $\{\DD_i=(x_i,x_{i+1}],\ x_i=i h,\ i=0,1,\cdots,n-1\}$, where $h=1/n$. We define the piecewise constant approximations of $\dot{W}$ by
\begin{align}\label{w-z}
\dot{W}^n(x)=\sum_{i=0}^{n-1}\frac{\chi_i(x)}{h}\int_{\DD_i}dW(y), 
\quad n\in \N,\ x\in \overline \DD.
\end{align}
where $\chi_i$ is the characteristic function of $\DD_i$. It is apparent that for each $n\in \N$, $\dot{W}^n\in L^2(\DD)$ a.s. 
However, we have the following identity which shows that 
$\E\left[\|\dot{W}^n\|^2\right]$ is unbounded as $h\rightarrow0$:
\begin{align}\label{wh-est}
\E\left[\|\dot{W}^n\|^2\right]=h^{2H-2},\quad \forall\ n\in \N.
\end{align}

The following estimate  will play an important role both in the proof of the existence of the weak solution of Eq. \eqref{ell} and in the error estimate of piecewise constant approximations.

\begin{lm}\label{dif-gri}
There exists $C=C(H)$ such that
\begin{align}\label{dif-gri0}
\sum_{i\neq j}^n\int_{\DD_i}\int_{\DD_j}|x-y|^{2H-2}dxdy\le Ch^{2H-1}.
\end{align}
\end{lm}

\begin{proof}By direct calculation, for $i,j\in \{0,1,\cdots,n-1\}$ and $i\neq j$,
\begin{align*}
\int_{\DD_i}\int_{\DD_j}|x-y|^{2H-2}dxdy=\frac{A_{i,j}(H)h^{2H}}{2H(1-2H)},
\end{align*}
where $A_{i,j}(H)=2|i-j|^{2H}-|i-j+1|^{2H}-|i-j-1|^{2H}$.
A simple calculation implies that
$\sum\limits_{i\neq j}A_{i,j}(H)=(n-n^{2H})/2$.
As a consequence,
\begin{align*}
\sum_{i\neq j}\int_{\DD_i}\int_{\DD_j}|x-y|^{2H-2}dxdy
=\frac{h^{2H}}{2H(1-2H)}\sum_{i\neq j}A_{i,j}(H)
=\frac{h^{2H}(n-n^{2H})}{H(1-2H)}\le \frac{h^{2H-1}}{H(1-2H)},
\end{align*}
which proves \eqref{dif-gri0} with $C=\frac{1}{H(1-2H)}$.
\end{proof} \\

Define the error between the two stochastic convolutions by $E^n$:
\begin{align}\label{en}
E^n(x):=\int_{\DD} G(x,y) dW(y)-\int_{\DD} G(x,y) dW^n(y),
\quad x\in \DD.
\end{align}
From  \eqref{w-z} we have 
\begin{align*}
\int_{\DD} G(x,y)dW^n(y)
=\int_{\DD}\left(\sum_{i=0}^{n-1} \frac{\chi_i(y)}{h}\int_{\DD_i}G(x,z)dz\right)dW(y).
\end{align*}
Then we can rewrite $E^n$ as
\begin{align*}
E^n(x)=\frac{1}{h}\sum_{i=0}^{n-1}\int_{\DD_i}\int_{\DD_i}(G(x,y)-G(x,z))dzdW(y).
\end{align*}
Next we use Lemma \ref{dif-gri} to derive an estimate for $E^n$.

\begin{prop}\label{en-est}
There exists a constant $C=C(H)$ such that
\begin{align}\label{en-est0}
\sup_{x\in \DD}\E\left[|E^n(x)|^2\right]
\le Ch^{2H+1}.
\end{align}
\end{prop}

\begin{proof}
Appyling It\^{o} isometry formula \eqref{ito}, we obtain
\begin{align}
\E\left[|E^n(x)|^2\right]
&=\frac{H(1-2H)}{2}\int_{\DD}\int_{\DD} \frac{|[G(x,y)-\widehat{G}(x,y)]-[G(x,z)-\widehat{G}(x,z)]|^2}{|y-z|^{2-2H}}dydz     \nonumber  \\
&\quad +H\int_{\DD} |G(x,y)-\widehat{G}(x,y)|^2(y^{2H-1}+(1-y)^{2H-1})dy
=:\frac{H(1-2H)}{2}\cdot I_1+H\cdot I_2. \label{eh2}
\end{align}
For $I_1$, we first split it into two parts as follows:
\begin{align}
I_1&=\frac{1}{h^2}\sum_{i\neq j}^{n-1}\int_{\DD_i}\int_{\DD_i}
\frac{\left|\int_{\DD_i} G(x,u)-G(x,y) du-\int_{\DD_j} G(x,v)-G(x,z) dv\right|^2}{|y-z|^{2-2H}}dydz  \nonumber  \\
&\quad + \sum_{i=0}^{n-1}\int_{\DD_i}\int_{\DD_i} 
\frac{|G(x,y)-G(x,z)|^2}{|y-z|^{2-2H}}dydz
=: I_{11}+I_{12}.\label{e1}
\end{align}
Applying H\"{o}lder's inequality and the estimate \eqref{dif-gri0} in Lemma \ref{dif-gri}, we get
\begin{align}\label{e11}
I_{11}
&{\color{blue} \le} \frac{1}{h^2}\sum_{i\neq j}^{n-1}\int_{\DD_i}\int_{\DD_j}\int_{\DD_i}\int_{\DD_j}\frac{|[G(x,u)-G(x,y)]-[G(x,v)-G(x,z)]|^2}{|y-z|^{2-2H}}dudvdydz \nonumber   \\
&\le \frac{2}{h^2}\sum_{i\neq j}^{n-1}\int_{\DD_i}\int_{\DD_j}\int_{\DD_i}\int_{\DD_j}\frac{|u-y|^2+|v-z|^2}{|y-z|^{2-2H}}dudvdydz   \nonumber \\
&\le 4h^2\sum_{i\neq j}^{n-1}\int_{\DD_i}\int_{\DD_j}|y-z|^{2H-2}dudvdydz
\le 4h^{2H+1}.
\end{align}
Since the Green's function is Lipschitz continuous,
\begin{align}\label{e12}
I_{12}
{\color{blue} 
\le \sum_{i=0}^{n-1}\int_{\DD_i}\int_{\DD_i} 
\frac{|G(x,y)-G(x,z)|^2}{|y-z|^{2-2H}}dydz}
\le \sum_{i=0}^{n-1}\int_{\DD_i}\int_{\DD_i} |y-z|^{2H} dydz
=\frac{2h^{2H+1}}{(2H+1)(2H+2)}.
\end{align}

Next we evaluate $I_2$. Since the Green's function $G$ is Lipschitz continuous,
\begin{align}\label{e22}
I_2
=\sum_{i=0}^{n-1}\int_{\DD_i}\left|\frac{1}{h}\int_{\DD_i}
G(x,u)-G(x,y) du\right|^2
\left( y^{2H-1}+(1-y)^{2H-1} \right) dy 
\le C h^2.
\end{align}
Combining \eqref{eh2}--\eqref{e22}, we obtain the desired estimate \eqref{en-est0}.
\end{proof} \\

For $\phi\in L^2(\DD)$, define
$K\phi:=\int_{\DD} G(\cdot,y)\phi(y)dy$.
We also denote $K\dot{W}:=\int_{\DD} G(\cdot,y)dW(y)$.
Set $f(u)=f(\cdot,u(\cdot))$. 
Then \eqref{mild} can be rewritten as
\begin{align}\label{mild-1}
u+Kf(u)=Kg+K\dot{W}.
\end{align}
 To prove the existence of a unique solution of Eq. \eqref{mild-1}, we need the following inequality which can be derived from the Poincar\'{e}'s inequality \eqref{poin} \citep[see, e.g.,][Lemma 2.4]{BP90}:
\begin{align}
(K\phi,\phi)\ge \gamma\|K\phi\|^2,\quad\forall\ \phi\in L^2(\DD). \label{poin2}
\end{align}

\begin{tm}
Let Assumption \ref{ap} hold.  
Eq. \eqref{ell} possesses a unique mild solution.
\end{tm}

\begin{proof}
We first prove the uniqueness. 
Suppose that $u$ and $v$ solve Eq. \eqref{mild}.
Then
\begin{align*}
u-v+K(f(u)-f(v))=0.
\end{align*}
Multiplying by $f(u)-f(v)$  on the above equation, we have
\begin{align*}
(u-v,f(u)-f(v))+(K(f(u)-f(v)),f(u)-f(v))=0.
\end{align*}
From the monotone type condition \eqref{one-lip} in Assumption \ref{ap} and \eqref{poin2} we deduce that
\begin{align*}
(\gamma-L)\|u-v\|^2 \le 0,
\end{align*}
which implies that $u=v$.

Next we prove the existence. 
The proof is for bounded $f$. 
The general case of $f$ satisfying the linear growth condition \eqref{lin-gro} follows from localization arguments in \citep[Theorem 2.5]{BP90}.
For each $n\in \N_+$, we consider the SPDE obtained by replacing  $\dot W$ with $\dot W^n$ in Eq. \eqref{ell}:
\begin{align}\label{un}
\begin{split}
-\frac{d^2}{dx^2}u^n+f(u^n)
&=g+\dot{W}^n \quad\text{in}\quad \DD,\\
u^n
&=0 \quad\qquad\ \text{on}\quad\partial \DD. 
\end{split}
\end{align}
The existence of a unique solution $u^n\in H^1_0(\DD)$ for Eq. \eqref{un} follows from the classical deterministic analysis.
Clearly, $u^n-u^m+K(f(u^n)-f(u^m))=K(\dot{W}^n-\dot{W}^m)$.
Multiplying by $f(u^n)-f(u^m)$, we obtain
\begin{align*}
(u^n-u^m,f(u^n)-f(u^m))+(K(f(u^n)-f(u^m)),f(u^n)-f(u^m))
=(K(\dot{W}^n-\dot{W}^m),f(u^n)-f(u^m)).
\end{align*}
It follows from the monotone type condition \eqref{one-lip} and Poincar\'e inequality \eqref{poin2} that
\begin{align}
(\gamma-L)\|u^n-u^m\|^2 \le (K(\dot{W}^n-\dot{W}^m),f(u^n)-f(u^m)+2\gamma(u^n-u^m)).
\end{align}
Since $\E\left[\|K(\dot{W}^n-\dot{W}^m)\|^2\right]$ tends to $0$ as $n,m\rightarrow\infty$ and $f$ is bounded, $\{u^n\}$ is a Cauchy sequence in $L^2(\DD\times \Omega)$. Hence there exists $u$ in $L^2(\DD\times \Omega)$ such that $u=\lim_{n\rightarrow\infty}u^n$. 
From the boundedness of $f$ and Assumption \ref{ap}, $f(u^n)\rightarrow f(u)$ in $L^2(\DD\times \Omega)$ as $n\rightarrow\infty$. The existence then follows from taking the limit in \eqref{un}.
\end{proof}

\section{Error estimates of piecewise constant approximations}
\label{sec3}

In this section, we estimate the error between the solution of Eq. \eqref{ell} and the solution of the approximate equation
\begin{align}\label{un0}
\begin{split}
-\frac{d^2}{dx^2}u^n+f(u^n)
&=g+\dot{W}^n \quad\text{in}\quad \DD,\\
u^n
&=0 \quad\qquad\ \text{on}\quad\partial \DD. 
\end{split}
\end{align}

Set $F^n=g+\dot{W}^n$. 
The variational formulation of Eq. \eqref{un0} is to find a $u^n\in H^1_0(\DD)$ such that a.s.
\begin{align}
\left(\frac{d}{dx}u_n,\frac{d}{dx}v\right)+(f(u_n),v)
=(F^n,v),\quad \forall\ v\in H^1_0(\DD). \label{var}
\end{align}

We first analyze the regularity and obtain a bound for $u^n$, which will play a key role in the error estimate of the finite element approximation for Eq. \eqref{un0} in Section \ref{sec4}.

\begin{tm}\label{uheat}
Let Assumption \ref{ap} hold. 
Eq. \eqref{var}, therefore Eq. \eqref{un0}, has a unique solution $u^n\in H^1_0(\DD)\cap H^2(\DD)$, a.s. Moreover, there exists a constant $C$ such that
\begin{align}
\E\left[\|u^n\|^2_2\right]
\le C h^{2H-2}.\label{euh}
\end{align}
\end{tm}

\begin{proof}
The existence of a unique solution $u^n\in H^1_0(\DD)$ a.s. follows from the classical deterministic arguments. 
To obtain \eqref{euh}, we first notice that Assumption \ref{ap}, the Poincar\'{e}'s inequality \eqref{poin} and Cauchy-Schwarz inequality yield that
\begin{align*}
\|F_n\|\cdot \|u_n\|
\ge (F_n, u_n)
=\Big\|\frac{d}{dx}u_n\Big\|^2+(f(u^n), u^n)
\ge (\gamma-L)\|u^n\|^2,
\end{align*}
from which we obtain
\begin{align*}
\|u^n\|\le\frac{1}{\gamma-L}\|F^n\|.
\end{align*}
Set $R^n=F^n-f(u^n)$. 
The linear growth condition \eqref{lin-gro} gives
\begin{align*}
\|R^n\|^2\le 4\beta^2
+\left(2+\frac{4\beta^2}{(\gamma-L)^2}\right)\|F^n\|^2.
\end{align*}
On the other hand, it follows from Eq. \eqref{un0} that $u^n\in H^2(\DD)$ and 
\begin{align*}
\|u^n\|_2^2\le C\|R^n\|^2
\end{align*} 
for some $C\in (0,\infty)$.
We conclude \eqref{euh} by combining the above estimates and \eqref{wh-est}.
\end{proof} \\

Next we estimate the error between the exact solution $u$ of Eq. \eqref{ell} and its approximation $u^n$ defined by Eq. \eqref{un0}. 
Recall that it follows from Definition \ref{df-mild} that $u$ and $u^n$ are the unique solutions of the following Hammerstein integral equations, respectively:
\begin{align}
u+Kf(u)=&Kg+K\dot{W}, \label{mildu}\\
u^n+Kf(u^n)=&Kg+K\dot{W^n}.\label{milduh}
\end{align}

\begin{tm}\label{uun}
Let Assumption \ref{ap} hold.
There exists a constant $C$ such that
\begin{align}\label{uun0}
\sqrt{\E\left[\|u-u^n\|^2\right]}\le C h^{\frac H2+\frac14}. 
\end{align}
Assume furthermore that $f$ is Lipschitz continuous with the Lipschitz constant $L<\gamma$, then
\begin{align}\label{uun1}
\sqrt{\E\left[\|u-u^n\|^2\right]} \le C h^{H+\frac 12}. 
\end{align}
\end{tm}

\begin{proof}Subtracting \eqref{milduh} from \eqref{mildu}, we obtain
\begin{align}
u(x)-u^n(x)+K(f(u)-f(u^n))=E^n. \label{u-uh}
\end{align}
In terms of the estimate \eqref{en-est0} of $E^n$ defined by \eqref{en} in Lemma \ref{en-est}, to prove \eqref{uun0}, it suffices to prove
\begin{align}\label{ueh}
\|u-u^n\|^2\le C\|E^n\|^2+\|E^n\|. 
\end{align}
To this end, we multiply  \eqref{u-uh} by $f(u)-f(u^n)$ to obtain
\begin{align*}
(u-u^n,f(u)-f(u^n))+(K(f(u)-f(u^n)),f(u)-f(u^n))=(E^n,f(u)-f(u^n)).
\end{align*}
The estimate \eqref{poin2} and the monotone type condition \eqref{one-lip} yield
 \begin{align}\label{uef}
-L\|u-u^n\|^2+\gamma\|K\big(f(u)-f(u^n)\big)\|^2\le \|E^n\|\cdot\|f(u)-f(u^n)\|.
\end{align}
Using the Young type inequality
$\|\phi+\psi\|^2\ge \epsilon\|\phi\|^2-\frac{2-\epsilon}{1-\epsilon}\|\psi\|^2$ with $\phi=u-u^n,\psi=-E^n$ and $\epsilon=\frac{L+\gamma}{2\gamma}$, we obtain
 \begin{align}
\|K\big(f(u)-f(u^n)\big)\|^2=\|u-u^n-E^n\|^2\ge \frac{L+\gamma}{2\gamma}\|u-u^n\|^2-\frac{3\gamma-L}{\gamma-L}\|E^n\|^2. \label{kf-kfn}
\end{align}
By the average inequality $a b\le \frac{L-\gamma}{4\beta} a^2+\frac{\beta}{L-\gamma}b^2$ and \eqref{lin-gro}, we obtain
\begin{align}
\|E^n\|\cdot \|f(u)-f(u^n)\|
\le \beta \|E^n\|(1+\|u-u^n)\|
\le \beta\|E^n\|+\frac{L-\gamma}{4}\|u-u^n\|^2+\frac{\beta^2}{L-\gamma}\|E^n\|^2. \label{enf}
\end{align}
Substituting \eqref{enf} and \eqref{kf-kfn} into \eqref{uef}, we deduce that
\begin{align*}
-L\|u-u^n\|^2+\frac{L+\gamma}{2}\|u-u^n\|^2-\frac{2(3\gamma-L)}{\gamma-L}\|E^n\|^2
\le \beta\|E^n\|+\frac{L-\gamma}{4}\|u-u^n\|^2+\frac{\beta^2}{L-\gamma}\|E^n\|^2,
\end{align*}
from which the desired estimate \eqref{ueh} follows.

Now assume that $f$ is Lipschitz continuous with the Lipschitz constant $L<\gamma$, then the term $\|E^n\|$ in \eqref{ueh} would disappear. 
In this case we achieve \eqref{uun1}.
\end{proof}

\section{Finite Element Approximations}
\label{sec4}

In this section, we consider the finite element approximations of Eq. \eqref{var} and establish an overall error estimate between the exact solution and its finite element approximations.

Let $V_h$ be the continuous piecewise linear finite element subspace of $H^1_0(\DD)$ with respect to the  quasi-uniform partition  $\{\DD_i\}_{i=0}^{n-1}$ given in  Section \ref{sec2}.
Then the finite element approximation of Eq. \eqref{var} is to find an $u^n_h\in V_h$ for each $n\in \N$ such that
\begin{align}\label{fem}
\left(\frac{d}{dx} u^n_h,\frac{d}{dx} v_h\right)+(f(u^n_h),v_h)
=(F^n,v_h),\quad \forall\ v_h\in V_h. 
\end{align}

\begin{tm}
Let Assumption \ref{ap} hold.
Eq. \eqref{fem} has a unique solution $u^n_h\in H^1_0(\DD)$, a.s. 
Moreover, there exists a constant $C$ such that
\begin{align}\label{uh}
\E\left[\|u^n_h\|_1^2\right]
\le C h^{2H-2}. 
\end{align}
\end{tm}

\begin{proof}
Following a similar argument as in the proof of Theorem \ref{uheat}, we have
\begin{align}\label{ufn}
\|u^n_h\|\le\frac{\|F^n\|}{\gamma-L}.
\end{align}
Define $R^n_h=F^n-f(u^n_h)$. The linear growth condition \eqref{lin-gro} together with \eqref{ufn} implies
\begin{align}\label{r-h}
\|R^n_h\|^2\le 4\beta^2
+\left( 2+\frac{4 \beta^2}{(\gamma-L)^2}\right) \|F^n\|^2.
\end{align}
Notice that $u^n_h$ is the solution of
\begin{align*}
\left(\frac{d}{dx} u^n_h,\frac{d}{dx} v_h\right)=(R^n_h,v_h),
\quad \forall\ v_h\in V_h, 
\end{align*}
from which we derive
\begin{align}
\|u^n_h\|_1^2\le C \|R^n_h\|^2. \label{uh-rh}
\end{align}
We conclude the estimate \eqref{uh} with \eqref{ufn}--\eqref{uh-rh} and \eqref{wh-est}.
\end{proof} \\

Next we derive an estimate between $u^n$ and $u^n_h$. For this purpose we introduce the Galerkin (or Ritz) projection operator $\RR_h: H^1_0(\DD)\rightarrow V_h$ defined by
\begin{align}\label{gal-pro}
\left(\frac{d}{dx} \RR_h w,\frac{d}{dx} v_h\right)
=\left(\frac{d}{dx} w,\frac{d}{dx} v_h\right),\quad  \forall\ v_h\in V_h,\ w\in H^1_0(\DD). 
\end{align}
It is well-known that there exists a constant $C$ such that \citep[see, e.g.,][Lemma 1.1]{Tho06}
\begin{align}\label{standard}
\|w-\RR_h w\|+h\Big\|\frac{d}{dx}(w-\RR_h w)\Big\|
\le C h^2\|w\|_2,\quad \forall\ w\in H^1_0(\DD)\cap H^2(\DD).
\end{align}

\begin{tm}\label{ununh}
Let Assumption \ref{ap} hold.
There exists a constant $C$ such that
\begin{align}\label{ununh0}
\sqrt{\E\left[\|u^n-u^n_h\|^2\right]}
\le C h^\frac{H+1}2. 
\end{align}
Assume furthermore that $f$ is Lipschitz continuous with the Lipschitz constant $L<\gamma$, then
\begin{align}\label{ununh1}
\sqrt{\E\left[\|u^n-u^n_h\|^2\right]}
\le C h^{H+1}. 
\end{align}
\end{tm}

\begin{proof}It follows from \eqref{var}, \eqref{fem} and \eqref{gal-pro} that
\begin{align}
\Big\|\frac{d}{dx}(\RR_hu^n-u^n_h)\Big\|^2
+(f(u^n)-f(u^n_h),\RR_hu^n-u^n_h)=0.
\end{align}
The Assumptions \ref{ap} and the average inequality $a\cdot b\le \frac{\gamma-L}{2\beta} a^2+\frac{\beta}{2(\gamma-L)}b^2$ with $a=\|u^n-u^n_h\|$ and $b=\|u^n-\RR_hu^n\|$ yield
\begin{align}\label{unpun}
\Big\|\frac{d}{dx}(\RR_hu^n-u^n_h)\Big\|^2  
\le \frac{\gamma+L}2\|u^n-u^n_h\|^2+\beta\|u^n-\RR_hu^n\|
+\frac{\beta^2}{2(\gamma-L)}\|u^n-\RR_hu^n\|^2. 
\end{align}
Applying the projection theorem, Poincar\'{e} inequality \eqref{poin} and the above inequality, we obtain
\begin{align}
\gamma\|u^n-u^n_h\|^2 
\le \frac{\gamma+L}2\|u^n-u^n_h\|^2+\beta\|u^n-\RR_hu^n\|
+\frac{\gamma+\beta^2}{2(\gamma-L)}\|u^n-\RR_hu^n\|^2
\end{align}
from which and \eqref{standard} we derive
\begin{align}
\|u^n-u^n_h\|^2
\le C(\|u^n-\RR_hu^n\|+\|u^n-\RR_hu^n\|^2)
\le C h^2\|u^n\|_2. \label{u-u}
\end{align}
The desired error estimate then follows  from \eqref{u-u} and Theorem \ref{uheat}.

Now assume that $f$ is Lipschitz continuous with the Lipschitz constant $L<\gamma$, then the term $\|u^n-\RR_hu^n\|$ in \eqref{unpun} would disappear. As a consequence,
\begin{align}
\|u^n-u^n_h\|^2
\le C \|u^n-\RR_hu^n\|^2
\le C h^4\|u^n\|_2^2
\le C h^{2H+2}.
\end{align}
This leads to the estimate \eqref{ununh1}.
\end{proof} 

\begin{rk}
We should not expect any estimate of $\E\left[\|\frac{d}{dx}(u^n-u^n_h)\|^2\right]$ with a positive order since $\E\left[\|u^n\|_2^2\right]=\mathcal O(h^{2H-2})$.
However, by the proof of Theorem \ref{ununh},
\begin{align*}
\E\left[\Big\|\frac{d}{dx}(\RR_hu^n-u^n_h)\Big\|^2\right]
\le C h^{H+1},
\end{align*}
which agrees with the property of super-convergence of finite element method.
\end{rk}

Combining Theorem \ref{uun} and Theorem \ref{ununh}, we derive the main result about the error estimate between the exact solution $u$ and finite element solution $u^n_h$ by the triangle inequality.

\begin{tm}
Under Assumption \ref{ap}, the error between the exact solution $u$ of Eq. \eqref{ell} and its finite element solution $u^n_h$ defined by \eqref{fem} satisfies
\begin{align}\label{uunh1}
\sqrt{\E\left[\|u-u^n_h\|^2\right]}
\le C h^{\frac H2+\frac14}.
\end{align}
Assume furthermore that $f$ is Lipschitz continuous with the Lipschitz constant $L<\gamma$, then
\begin{align}\label{uunh2}
\sqrt{\E\left[\|u-u^n_h\|^2\right]}
\le C h^{H+\frac12}.
\end{align}
\end{tm}

\begin{proof}
The estimates \eqref{uunh1} and \eqref{uunh2} follows from \eqref{uun0}, \eqref{uun1} in Theorem \ref{uun} and 
\eqref{ununh0}, \eqref{ununh1} in Theorem \ref{ununh}. 
\end{proof}

\section{Conclusions}
\label{sec5}

In this paper we developed the Galerkin finite element method for the boundary value problem of a one dimensional second order SDE driven by an fBm. 
The Hurst index $H$ of the fBm is assumed to be equal to or less than $1/2$. 
We proved that, with continuous piecewise linear finite elements, the mean square convergence rate of the finite element approximations in the case of Lipschitz coefficient is $\mathcal O(h^{H+1/2})$, which is consistent with the existing result for white noise \citep[see, e.g.][]{ANZ98, GM06}.
In a separate work \citep[see][]{CHL16}, we have obtained strong convergence rate of finite element approximations for one dimensional time dependent SPDEs, including nonlinear stochastic heat equation and stochastic wave equation, driven by a fractional Brownian sheet which is temporally white and spatially fractional with $H\le 1/2$.
In future work, we plan to study the optimal convergence order of finite element approximations for SPDEs \eqref{ell} in high dimensional domains driven by a fractional Brownian sheet with $H\le 1/2$.


\bibliographystyle{IMANUM-BIB}
\bibliography{IMANUM-refs}

\end{document}